\RequirePackage{fix-cm}
\documentclass[smallcondensed]{svjour3}     
\smartqed  
\usepackage{graphicx}
\usepackage{times,a4wide,mathrsfs}
\usepackage{amsmath}
\usepackage{amsfonts}

\newcommand{\C}{\mathbb{C}}

\newcommand{\QQ}{\mathbb{Q}}
\newcommand{\NN}{\mathbb{N}}

\newcommand{\OO}{\mathcal O}

\newcommand{\id}{\hbox{id}}

\newcommand{\gr}{\hbox{Gr}}
\newcommand{\wt}{\widetilde}
\newcommand{\ima}{\hbox{Im}}

\newtheorem{convention}{Conventions}

\newtheorem{nonumbering}{Theorem}

\begin{document}

\title{On a multiplicative version of Mumford's theorem
}


\author{Robert Laterveer 
}


\institute{CNRS - IRMA, Universit\'e de Strasbourg \at
              7 rue Ren\'e Descartes \\
              67084 Strasbourg cedex\\
              France\\
              \email{laterv@math.unistra.fr}           
           }

\date{Received: date / Accepted: date}

\maketitle

\begin{abstract} A theorem of Esnault, Srinivas and Viehweg asserts that if the Chow group of $0$--cycles of a smooth complete complex variety decomposes, then the top--degree coherent cohomology group decomposes similarly. In this note, we prove a similar statement for Chow groups of arbitrary codimension, provided the variety satisfies the Lefschetz standard conjecture.

\end{abstract}


\section{Introduction}

Since Mumford's famous 1969 paper \cite{M}, it is well--known that the Chow group of $0$--cycles $A^nX$ on a complex variety $X$ influences the cohomology group $H^n(X,\QQ)$:

\begin{nonumbering}[Mumford \cite{M}] Let $X$ be a smooth complete variety of dimension $n$ defined over $\C$. Suppose that $A^nX_{\QQ}$ is supported on a divisor. Then $H^n(X,\QQ)$ is supported on a divisor, in particular $H^n(X,\OO_X)=0$.
\end{nonumbering}

In the 1992 paper \cite{ESV}, Esnault, Srinivas and Viehweg study the multiplicative behaviour of the Chow ring $A^\ast X$ versus the multiplicative behaviour of various cohomology rings associated to $X$. We now state the part of their result that is relevant to us. For a given partition 
   $n=n_1+\cdots+n_r$
   (with $n_i\in\NN_{>0}$),
let us consider the following properties:

\vskip0.5cm
\noindent
(P1) There exists a Zariski open $V\subset X$, such that intersection product induces a surjection
  \[   A^{n_1}V_{\QQ}\otimes A^{n_2}V_{\QQ}\otimes\cdots\otimes A^{n_r}V_{\QQ}\ \to\ A^nV_{\QQ}\ ;\]

  \noindent
  (P2) There exists a Zariski open $V\subset X$, such that cup product induces a surjection
  \[ H^{n_1}(V,\QQ)\otimes H^{n_2}(V,\QQ)\otimes\cdots\otimes H^{n_r}(V,\QQ)\ \to\ H^n(V,\QQ)/N^1\ \]
  (here $N^\ast$ denotes the coniveau filtration);
  
  \noindent
  (P3) Cup product induces a surjection
    \[ H^{n_1}(X,\OO_X)\otimes H^{n_2}(X,\OO_X)\otimes\cdots\otimes H^{n_r}(X,\OO_X)\ \to\ H^n(X,\OO_X)\ .\]

\vskip0.5cm

In these terms, what Esnault, Srinivas and Viehweg prove is the following:

\begin{nonumbering}[Esnault--Srinivas--Viehweg \cite{ESV}] Let $X$ be a smooth complete variety of dimension $n$ over $\C$. Then (P1) implies (P3), and (P2) implies (P3).
\end{nonumbering}

The implication from (P1) to (P3) is a kind of multiplicative variant of Mumford's theorem, and the proof in \cite{ESV} is motivated by Bloch's proof of Mumford's theorem using a ``decomposition of the diagonal'' argument (\cite{B}, \cite{B0}, cf. also \cite{BS}). As noted in \cite[remark 2]{ESV}, the generalized Hodge conjecture would imply that (P2) and (P3) are equivalent.\footnote{It is somewhat frustrating that it is not known unconditionally whether (P1) implies (P2), i.e. without assuming the generalized Hodge conjecture. Apparently Esnault, Srinivas and Viehweg had claimed to prove this in an earlier version of their paper, but the argument was found to be incomplete \cite[remark 2]{ESV}.}

In this note, we show that the Esnault--Srinivas--Viehweg theorem can be extended from $0$--cycles to arbitrary Chow groups. This is possible provided the variety $X$ satisfies the Lefschetz standard conjecture $B(X)$ (this is analogous to \cite{yet}, where I extended Mumford's theorem from $0$--cycles to arbitrary Chow groups, provided $B(X)$ holds):

\begin{nonumbering}[(=theorem \ref{main})] Let $X$ be a smooth projective variety of dimension $n$ over $\C$ that satisfies $B(X)$. Suppose there exists a Zariski open $V\subset X$, and $j=j_1+\cdots+ j_r$ with $j_i\in\NN_{>0}$ such that intersection product induces a surjection
  \[  A^{j_1}V_{\QQ}\otimes A^{j_2}V_{\QQ}\otimes\cdots\otimes A^{j_r}V_{\QQ}\ \to\ A^jV_{\QQ}\ .\]
Then cup product induces a surjection
  \[ H^{j_1}(X,\OO_X)\otimes H^{j_2}(X,\OO_X)\otimes\cdots\otimes H^{j_r}(X,\OO_X)\ \to\ H^j(X,\OO_X)\ .\]
\end{nonumbering}

The proof of this theorem, which is very similar to the proof given by Esnault--Srinivas--Viehweg in \cite{ESV}, is an exercise in using the meccano of correspondences and the Bloch--Srinivas formalism.

It seems natural to wonder whether the converse to theorem \ref{main} might perhaps be true (this would be a multiplicative variant of Bloch's conjecture). In \cite{moi}, I prove this converse implication in some special cases for $0$--cycles (i.e. $j=n$); the converse implication for $j\not=n$ appears to be more difficult.

\begin{convention} In this note, the word {\sl variety\/} will refer to a quasi--projective irreducible algebraic variety over $\C$, endowed with the Zariski topology. A {\sl subvariety\/} is a (possibly reducible) reduced subscheme which is equidimensional. The Chow group of $j$--dimensional algebraic cycles on $X$ with $\QQ$--coefficients modulo rational equivalence is denoted $A_jX$; for $X$ smooth of dimension $n$ the notations $A_jX$ and $A^{n-j}X$ will be used interchangeably. Caveat: note that what we denote $A^jX$ is elsewhere often denoted $CH^j(X)_{\QQ}$.
In an effort to lighten notation, we will often write $H^jX$ or $H_jX$ to indicate singular cohomology $H^j(X,\QQ)$ resp. Borel--Moore homology $H_j(X,\QQ)$.

For basics concerning algebraic cycles and their functorial behaviour, the curious reader is invited to consult \cite{F}. For the formalism of correspondences, cf. \cite{Sch}, \cite{MNP}.
\end{convention}

\section{Preliminary}

Let $X$ be a smooth projective variety of dimension $n$, and $h\in H^2(X,\QQ)$ the class of an ample line bundle. The hard Lefschetz theorem asserts that the map
  \[  L^{n-i}\colon H^i(X,\QQ)\to H^{2n-i}(X,\QQ)\]
  obtained by cupping with $h^{n-i}$ is an isomorphism, for any $i< n$. One of the standard conjectures asserts that the inverse isomorphism is algebraic.

\begin{definition}[Lefschetz standard conjecture]{} Given a variety $X$, we say that $B(X)$ holds if for all ample $h$, and all $i<n$ the isomorphism 
  \[  (L^{n-i})^{-1}\colon 
  H^{2n-i}(X,\QQ)\stackrel{\cong}{\rightarrow} H^i(X,\QQ)\]
  is induced by a correspondence.
 \end{definition}  
 
 \begin{remark} It is known that $B(X)$ holds for the following varieties: curves, surfaces, abelian varieties \cite{K0}, \cite{K}, threefolds not of general type \cite{Tan}, varieties motivated by a surface in the sense of Arapura \cite{A2} (this includes the Hilbert schemes of $0$--dimensional subschemes of surfaces \cite[Corollary 7.5]{A2}), 
 $n$--dimensional varieties $X$ which have $A_i(X)_{}$ supported on a subvariety of dimension $i+2$ for all $i\le{n-3\over 2}$ \cite[Theorem 7.1]{V}, $n$--dimensional varieties $X$ which have $H_i(X)=N^{\llcorner {i\over 2}\lrcorner}H_i(X)$ for all $i>n$ \cite[Theorem 4.2]{V2}, products and hyperplane sections of any of these \cite{K0}, \cite{K}.
 \end{remark}

It is known that $B(X)$ implies that the K\"unneth components
  \[  \pi_j\in H^{2n-j}(X)\otimes H^j(X)\subset H^{2n}(X\times X)\]
  of the diagonal $\Delta\subset X\times X$ are algebraic \cite{K0}, \cite{K}. Moreover, these K\"unneth components satisfy the following property:
  
  \begin{lemma}\label{kunneth} Let $X$ be a smooth projective variety satisfying $B(X)$, and let $h\in H^2(X)$ be the class of an ample line bundle.
  For any $j\le n$, there exists a cycle $P_j\in A^j(X\times X)$ such that
    \[   \pi_j=(\tau\times\id)_\ast(\tau\times\id)^\ast(P_j)\ \ \in H^{2n}(X\times X)\ ,\]
    where $\tau\colon Y_j\to X$ denotes the inclusion of a dimension $j$ complete intersection of class $[Y_j]=h^{n-j}$.
  \end{lemma}

\begin{proof} As mentioned above, $B(X)$ ensures that $\pi_j$ is algebraic \cite{K0}, \cite{K}. Consider now the isomorphism
  \[   L^{n-j}\times\id\colon\ \ H^jX\otimes H^jX\ \xrightarrow{\cong}\ H^{2n-j}X\otimes H^jX\ \]
  (here we tacitly identify both sides with their images in $H^\ast(X\times X)$).
  
  Since we have $B(X)$, there exists a correspondence, say $Q\in A^{j}(X\times X)$, such that
    \[  (L^{n-j}\times\id)   (Q\times\id)_\ast=\id\colon\ \ H^{2n-j}X\otimes H^jX\ \to\ H^{2n-j}X\otimes H^jX\ .\]
    Since $\pi_j$ is algebraic,
      \[ P_j:= (Q\times\id)_\ast (\pi_j)\in A^j(X\times X)\]
      is still algebraic, and has the requested property.
\end{proof}

\begin{remark} Lemma \ref{kunneth} implies in particular that for a variety satisfying $B(X)$, the K\"unneth component $\pi_j$ is represented by an algebraic cycle contained in $Y_j\times X$, for a dimension $j$ complete intersection $Y_j$. This was also proven in \cite{KMP} (and independently in \cite[proof of theorem 3.1]{yet}, as I wasn't aware of the Kahn--Murre--Pedrini reference at the time).
\end{remark}

\section{Main}

We now prove the main theorem of this note:

\begin{theorem}\label{main} Let $X$ be a smooth projective variety of dimension $n$ over $\C$ that satisfies $B(X)$. Suppose there exists a Zariski open $V\subset X$, and $j=j_1+\cdots+ j_r$ with $j_i\in\NN_{>0}$ such that intersection product induces a surjection
  \[  A^{j_1}V_{}\otimes A^{j_2}V_{}\otimes\cdots\otimes A^{j_r}V_{}\ \to\ A^jV_{}\ .\]
Then cup product induces a surjection
  \[ H^{j_1}(X,\OO_X)\otimes H^{j_2}(X,\OO_X)\otimes\cdots\otimes H^{j_r}(X,\OO_X)\ \to\ H^j(X,\OO_X)\ .\]
\end{theorem}

\begin{proof} Since $B(X)$ holds, it follows from lemma \ref{kunneth}
that the K\"unneth component $\pi_j$ can be written
    \[   \pi_j=(\tau\times\id)_\ast(\tau\times\id)^\ast(P_j)\ \ \in H^{2n}(X\times X)\ \]
    for some $P_j\in A^j(X\times X)$,
    where $\tau\colon Y_j\to X$ is the inclusion of a dimension $j$ complete intersection.
    
 Applying the Bloch--Srinivas argument, in the form of proposition \ref{BS} below, to the cycle $P_j\in A^j(X\times X)$, we find a decomposition
   \[   P_j=C_1\cdot \ldots \cdot C_r+\Gamma_1+\Gamma_2\ \ \in A^j(X\times X)\ ,\]
   where $\Gamma_1, \Gamma_2$ are supported on $D\times X$, resp. on $X\times D$, for some divisor $D\subset X$. This induces a decomposition of the K\"unneth component
   \[  \begin{split}\pi_j&=(\tau\times\id)_\ast(\tau\times\id)^\ast(C_1\cdot \ldots \cdot C_r)+
     (\tau\times\id)_\ast(\tau\times\id)^\ast   (\Gamma_1+\Gamma_2)\\
                                 &=(\tau\times\id)_\ast(\tau\times\id)^\ast(C_1\cdot \ldots \cdot C_r)+
        \Gamma_1^\prime+\Gamma_2^\prime      \ \ \in H^{2n}(X\times X)\ ,\end{split}\]
 where $\Gamma_2^\prime$ is still supported on $X\times D$, and $\Gamma_1^\prime$ is supported on $Z\times X$, for some $Z\subset X$ of dimension $j-1$ (indeed, the general complete intersection $Y_j$ will be in general position with respect to $D$; we then define $Z$ to be $D\cap Y_j$).
 
 Now we consider the action of $\pi_j$ on $H^j(X,\OO_X)$. Since $H^j(X,\OO_X)=\gr^0_F H^j(X,\C)$ (where $F$ is the Hodge filtration), $\pi_j$ acts as the identity on $H^j(X,\OO_X)$. On the other hand, it is clear that
   \[ (\Gamma_1^\prime)_\ast H^j(X,\OO_X)=0\]
   (by lemma \ref{comp2}, the action of $\Gamma_1^\prime$ factors over $\gr^0_F H^j(Z,\C)$, which is $0$ for dimension reasons), and also that
   \[ (\Gamma_2^\prime)_\ast H^j(X,\OO_X)=0\]
   (by lemma \ref{comp2}, the action of $\Gamma_2$ factors over $\gr^{-1}_F H^{j-2}(\wt{D},\C)=0$, where $\wt{D}$ is a resolution of singularities of $D$).
   To finish the argument, it only remains to analyze the action
   \[   \bigl((\tau\times\id)_\ast(\tau\times\id)^\ast(C_1\cdot \ldots \cdot C_r)\bigr)_\ast\colon\ \ H^j(X,\OO_X)\ \to\ H^j(X,\OO_X)\ .\]
   
   Using lemmas \ref{comp} and \ref{comp2}, we find an inclusion
   \[   \bigl((\tau\times\id)_\ast(\tau\times\id)^\ast(C_1\cdot \ldots \cdot C_r)\bigr)_\ast H^j(X,\OO_X)     \ \subset\ \bigl(C_1\cdot \ldots \cdot C_r\bigr)_\ast \gr_F^{n-j} H^{2n-j}(X,\C)\ .\]
   
   Using lemma \ref{esv}, we find that
   \[   \bigl(C_1\cdot \ldots \cdot C_r\bigr)_\ast \gr_F^{n-j} H^{2n-j}(X,\C)\subset  \ima\bigl(H^{j_1}(X,\OO_X)\otimes\cdots\otimes H^{j_r}(X,\OO_X)\ \to\ H^j(X,\OO_X)\bigr) \ ,\]
   and so we are done.
  \end{proof}

 \begin{proposition}[Bloch--Srinivas style]\label{BS} Let $X$ be a smooth projective variety of dimension $n$. Suppose there exists a Zariski open $V\subset X$, and $j=j_1+\cdots+ j_r$ with $j_i\in\NN_{>0}$ such that intersection product induces a surjection
  \[  A^{j_1}V_{}\otimes A^{j_2}V_{}\otimes\cdots\otimes A^{j_r}V_{}\ \to\ A^jV_{}\ .\]
  Then for any $a\in A^j(X\times X)$, there exists a decomposition
  \[   a= C_1\cdot\ldots\cdot C_r+\Gamma_1+\Gamma_2\ \ \in A^j(X\times X)\ ,\]
  where $C_i\in A^{j_i}(X\times X)$, and $\Gamma_1, \Gamma_2$ are supported on $D\times X$ (resp. on $X\times D$), for some divisor $D\subset X$.
  \end{proposition}
  
  \begin{proof} To be sure, this is a variant of the argument of \cite{BS}, exploiting the fact that $\C$ is a universal domain.
  Let $D_1\subset X$ denote the complement of $V$. Taking the smallest possible field of definition, we can suppose everything ($X$, $V$ and the cycle $a$) is defined over a field $k\subset\C$ which is finitely generated over its prime subfield. Then the inclusion $k(X)\subset\C$ (which comes from $\C$ being a universal domain) induces an injection
  \[  A^j(X_{k(X)})\ \to\ A^j(X_{\C})\ \]
  \cite[Appendix to Lecture 1]{B}. On the other hand,
  \[ A^j(X_{k(X)})_{}=   \varinjlim A^j(X\times U)_{}\ ,\]
 where the limit is taken over opens $U\subset X$ \cite[Appendix to Lecture 1]{B}.   
 
 Given the cycle $a\in A^j(X\times X)$, consider the restriction
   \[ a_{restr}\ \ \in A^j(X_{k(X)})\ .\]
   The assumption implies there exist cycles $c_i\in A^{j_i}(X_{\C})$ such that
   \[  a_{restr}=c_1\cdot\ldots\cdot c_r+a_0\ \ \in A^j(X_{\C})\ ,\]
   where $a_0$ is supported on $D_1$. Now, we extend $k$ so that the cycles $c_i$ are also defined over $k$ (and $k$ is still finitely generated over its prime subfield, so that $k(X)\subset\C$). Then using the injection $A^j(X_{k(X)})\ \to\ A^j(X_{\C})$ cited above, we obtain the decomposition
   \[  a_{restr}=c_1\cdot\ldots\cdot c_r+a_0\ \ \in A^j(X_{k(X)})\ .\]  
   Let $C_i\in A^{j_i}(X\times X)$ be any cycle restricting to $c_i$, and let $\Gamma_1$ be any cycle restricting to $a_0$. Then using the limit property cited above, we find that the difference
   \[ a-C_1\cdot\ldots\cdot C_r-\Gamma_1   \ \ \in A^j(X\times X)\]
   restricts to $0$ in $A^j(X\times U)$, for some open $U\subset X$. This means there exists a divisor $D_2\subset X$ and a cycle $\Gamma_2$ supported on $D_2$ such that
   \[   a= C_1\cdot\ldots\cdot C_r+\Gamma_1+\Gamma_2\ \ \in A^j(X\times X)\ .\]
   Taking $D$ a divisor containing both $D_1$ and $D_2$, this proves the proposition.
   \end{proof}
  
  \begin{lemma}\label{comp} Let $f\colon Y\to X$ be a proper morphism of smooth projective varieties, where $\dim X=n$ and $\dim Y=m$. Let $C\in A^j(X\times X)$.
  Then
    \[  \bigl( (f\times\id)^\ast C\bigr)_\ast = C_\ast f_\ast\colon\ \ H^iY\ \to\ H^{i+2(j-m)}X\ .\]
     \end{lemma}
  
  \begin{proof} This is purely formal, and surely well--known. Let $p_1, p_2\colon X\times X\to X$ denote projection on the first (resp. second) factor. Let $q_1, q_2$ denote projections from $Y\times X$ to $Y$ (resp. to $X$).
  For $a\in H^iY$, we have
    \[\begin{split} \bigl( (f\times\id)^\ast C\bigr)_\ast(a)&=   (q_2)_\ast \bigl( (q_1)^\ast(a)\cdot (f\times\id)^\ast (C)\bigr)\\
                    &= (p_2)_\ast \bigl( (f\times\id)_\ast ((q_1)^\ast(a)\cdot (f\times\id)^\ast(C))  \bigr) \\
                    &=(p_2)_\ast \bigl( (f\times\id)_\ast (q_1)^\ast(a)\cdot C\bigr)\\
                    &=(p_2)_\ast \bigl( (p_1)^\ast f_\ast (a)\cdot C\bigr)=: C_\ast f_\ast(a)\ \ \in H^{i+2(j-m)}X\ .\\
                    \end{split}\]
    \end{proof}

\begin{lemma}\label{comp2} Let $f\colon Y\to X$ be as in lemma \ref{comp}. Let $D\in A^j(Y\times X)$. Then
  \[   \bigl( (f\times\id)_\ast D\bigr)_\ast = D_\ast f^\ast\colon\ \ H^iX\ \to\ H^{i+2(j-m)}X\ .\]
  \end{lemma}
  
  \begin{proof} Just as lemma \ref{comp}, this is surely well--known. Keeping the notation of lemma \ref{comp}, for $b\in H^iX$ we have
    \[ \begin{split} \bigl( (f\times\id)_\ast D\bigr)_\ast (b)  &= (p_2)_\ast \bigl( (p_1)^\ast(b)\cdot (f\times\id)_\ast (D)\bigr)\\
                                                      &= (p_2)_\ast \bigl(   (f\times\id)_\ast \bigl( (f\times\id)^\ast (p_1)^\ast(b)\cdot D\bigr)\bigr)\\
                                                      &=(p_2)_\ast (f\times\id)_\ast \bigl(  (q_1)^\ast f^\ast(b)\cdot D\bigr)\\
                                                      &= (q_2)_\ast \bigl(  (q_1)^\ast f^\ast(b)\cdot D\bigr)=: D_\ast f^\ast(b)\ \ \in H^{i+2(j-m)}X\ .  \\
                                                         \end{split}\]
  \end{proof}

  \begin{lemma}[Esnault--Srinivas--Viehweg \cite{ESV}]\label{esv} Let $X$ be a smooth projective variety of dimension $n$. Let $C_i\in A^{j_i}(X\times X), i=1,\ldots,r$, with $j=j_1+\cdots+j_r$. Then
    \[   (C_1\cdot\ldots\cdot C_r)_\ast \gr_F^{n-j}H^{2n-j}(X,\C)\subset\ima\Bigl( H^{j_1}(X,\OO_X)\otimes\cdots\otimes H^{j_r}(X,\OO_X)\ \to\ H^j(X,\OO_X)\Bigr)\ .\]  
  \end{lemma}
  
  \begin{proof} This is shamelessly plagiarized from \cite{ESV}, who prove the $j=n$ case. Let
  $ C\in A^j(X\times X)$
  be any correspondence.
    The crucial observation is that the action
    \[  C_\ast\colon\ \ \gr_F^{n-j}H^{2n-j}(X,\C)\ \to\ \gr_F^0 H^j(X,\C)\]
  only depends on the image of $C$ under the composite map
  \[ \iota\colon\ \   A^j(X\times X)\ \to\  H^{2j}(X\times X,\C)\ \to\ H^jX\otimes H^jX\ \to\ \gr_F^j H^jX\otimes \gr_F^0 H^jX\ \]
  (Here the second map is given by the K\"unneth decomposition, and the last map is induced by projection on the appropriate summands of the Hodge decomposition).
  Indeed, suppose $C\in A^j(X\times X)$ is such that $\iota(C)=0$, i.e. the K\"unneth part of type $H^jX\otimes H^jX$ of $C$ is contained in 
     \[ \bigoplus_{i<j}  \gr_F^i H^jX\otimes \gr_F^{j-i} H^jX\ \subset\  \gr_F^j (H^jX\otimes H^jX)\ .\]
     Then, for $a\in \gr_F^{n-j}H^{2n-j}(X,\C)$ we find that 
     \[  (p_1)^\ast(a)\cdot C\ \ \in \bigoplus_{i<j} \gr_F^{n-j+i} H^{2n}X\otimes \gr_F^{j-i}H^jX=0\ ,\]
     and hence
     \[  C_\ast (a)=0\ \ \in H^j(X,\C)\ .\]
   Next, we apply this observation to
     \[ C=C_1\cdot\ldots\cdot C_r\ \ \in A^j(X\times X)\ ,\]
     with $C_i\in A^{j_i}X$.     
      The Hodge decomposition then gives that
     \[  \begin{split}\iota(C)=\iota(C_1)\cdot\ldots\cdot\iota(C_r)\in\ &\ima\bigl(\gr_F^{j_1} H^{j_1}X\otimes\cdots\otimes \gr_F^{j_r} H^{j_r}X\bigr)\otimes  \ima\bigl(\gr_F^{0} H^{j_1}X\otimes\cdots\otimes \gr_F^{0} H^{j_r}X\bigr)\\
                          &  \ \subset\    \gr_F^j H^jX\otimes \gr_F^0 H^jX\ .\\
                          \end{split}\]
     This proves the lemma: suppose
     \[ \iota(C)=\sum_k C^k_{left}\otimes C^k_{right}\ \ \in \gr_F^j H^jX\otimes \gr_F^0 H^jX\ .\]
     Then reasoning as above, we find that
     \[  (\iota(C))_\ast(a)= \sum_k (p_2)_\ast \bigl( (a\cup C^k_{left})\otimes C^k_{right} \bigr)=\sum_k \alpha_k C^k_{right}\ \ \in \gr_F^0 H^jX\ ,\]
     where the $\alpha_k$ are complex numbers (this is because $H^{2n}X$ is one--dimensional and generated by the class of a point). 
     \end{proof}

\begin{remark} It is mainly the contrapositive of theorem \ref{main} that is useful (this is another remark made in \cite{ESV} for their theorem). Indeed, suppose $X$ and $j=j_1+\ldots+j_r$ are such that
  \[   H^{j_1}(X,\OO_X)\otimes H^{j_2}(X,\OO_X)\otimes\cdots\otimes H^{j_r}(X,\OO_X)\ \to\ H^j(X,\OO_X)\ \]
is {\em not\/} surjective (for example, because
  \[ \prod_{i=1}^r \dim H^{j_i}(X,\OO_X)< \dim H^j(X,\OO_X)\ ).\]  
  Then by theorem \ref{main}, likewise
  \[     A^{j_1}X_{}\otimes \cdots\otimes A^{j_r}X_{}\ \to\ A^jX_{}\ \]
  fails to be surjective (and the same holds for any open $V\subset X$).
  \end{remark}

\begin{acknowledgements}
This note is the fruit of the Strasbourg 2014---2015 ``groupe de travail'' based on the monograph \cite{Vo}. I want to thank all the participants of this groupe de travail for the very pleasant and stimulating atmosphere, and their interesting lectures. Furthermore, many thanks to Yasuyo, Kai and Len for providing a wonderful working environment at home in Schiltigheim. 
Special thanks to Kai for all the bicycle trips made together this summer.
\end{acknowledgements}



\end{document}